\documentclass[11pt]{article}

 \usepackage{amsmath,amsthm,amssymb}
 \usepackage{makeidx,epsfig,lscape}
 \usepackage{color,colortbl}
 \usepackage{fancyhdr}
 \usepackage{times}
 \usepackage{xcolor,pict2e}
 \usepackage[latin1]{inputenc}
 
\newcommand{\R}{\mathbb R}
\newcommand{\N}{\mathbb N}
\newcommand{\F}{\mathcal F}

\newcommand{\Pv}{\mathbb P}

 \renewcommand{\headrulewidth}{0pt}
 \renewcommand{\footrulewidth}{0.5pt}

 \definecolor{myaqua}{rgb}{0.0,0.5,0.55}
 \definecolor{lightaqua}{rgb}{0.75,0.95,0.95}

 \usepackage[colorlinks = true,
            linkcolor = myaqua,
            urlcolor  = blue,
            citecolor = red,
            pdfpagemode=UseOutlines,
            bookmarksnumbered=true,
            pdfpagelabels,
            breaklinks]{hyperref}

\usepackage{caption}
\usepackage{floatrow}

\newtheorem{theorem}{Theorem}
\newtheorem{prop}[theorem]{Proposition}
\newtheorem{lem}{Lemma}
\newtheorem{coro}{Corollary}
\newtheorem{defn}{Definition}[section]
\newtheorem{rem}{Remark}[section]
\def\lin#1#2{\textcolor[rgb]{0.6,0.6,0.6}{\vspace*{#1mm} \hrule
   height 3 pt \vspace*{#2mm}}}
\def\bt{\begin{tabular}}
\def\et{\end{tabular}}
\def\and{\mbox{ and }}

\def\1{{\bf 1}}

 \def\boxx#1#2#3#4#5{
 {\linethickness{#4pt}\put(#1,#5){\color{myaqua}{\line(1,0){#3}}}}
 \multiput(#1,#2)(0,#4){2}{\line(1,0){#3}}
 \multiput(#1,#2)(#3,0){2}{\line(0,1){#4}}
  }

\begin{document}


 $\mbox{ }$

 \vskip 12mm

{ 

{\noindent{\Large\bf\color{myaqua}
  New contributions to the study of stochastic processes of the class \texorpdfstring{$(\Sigma)$}{sigma} }} 
%
\\[6mm]
{\bf Fulgence EYI OBIANG$^1$, Octave MOUTSINGA$^2$ and Youssef OUKNINE$^3$}}
\\[2mm]
{ 
 $^1$URMI Laboratory, Département de Mathématiques et Informatique, Faculté des Sciences, Université des Sciences et Techniques de Masuku, Franceville, Gabon 
  \\
Email: \href{mailto:feyiobiang@yahoo.fr}{\color{blue}{\underline{\smash{feyiobiang@yahoo.fr}}}}\\[1mm]
$^2$ URMI Laboratory, Département de Mathématiques et Informatique, Faculté des Sciences, Université des Sciences et Techniques de Masuku, Franceville, Gabon \\
\href{mailto:octavemoutsing-pro@yahoo.fr}{\color{blue}{\underline{\smash{octavemoutsing-pro@yahoo.fr}}}}\\[1mm]
$^3$LIBMA Laboratory, Department of Mathmatics, Faculty of Sciences Semlalia, Cadi Ayyad University, P.B.O. 2390 Marrakech, Morocco and Hassan II Academy of Sciences and Technologies, Rabat, Morocco
\\Email:\href{mailto:ouknine@ucam.ac.ma}{\color{blue}{\underline{\smash{ouknine@ucam.ac.ma}}}}\\[1mm]
\\[1mm]
\lin{5}{7}

 {  
 {\noindent{\large\bf\color{myaqua} Abstract}{\bf \\[3mm]
 \textup{
In this paper, we contribute to the study of the class $(\Sigma)$. In the first part of the paper, we provide  new ways to characterize stochastic processes of the above mentioned class and we derive some new properties. For instance, we prove that a stochastic process $X$ is an element of the class $(\Sigma)$ if, and only if, its absolute value is equal to absolute value of some martingale $M$. In the second part, we study in particular, stochastic processes of the class $(\Sigma)$ which vanish on the zero set of a given Brownian motion. More precisely, we provide a characterization theorem and methods dealing with such stochastic processes.  
 }}} 
 \\[4mm]
 {\noindent{\large\bf\color{myaqua} Keywords}{\bf \\[3mm]
Class $(\Sigma)$; Skorohod reflection  equation; Relative martingales; Honest time; Azema submartingale; Zeros of Brownian motion; Balayage formula; Enlargement of filtrations; Skew Brownian motion; Brownian local time. 
}}}\\[4mm]{\noindent{\large\bf\color{myaqua} MSC: }{\color{blue} 60G07; 60G20; 60G46; 60G48}}
\lin{3}{1}

\renewcommand{\headrulewidth}{0.5pt}
\renewcommand{\footrulewidth}{0pt}

 \pagestyle{fancy}
 \fancyfoot{}
 \fancyhead{} 
 \fancyhf{}
 \fancyhead[RO]{\leavevmode \put(-90,0){\color{myaqua}F. EYI OBIANG et al} \boxx{15}{-10}{10}{50}{15} }
 \fancyfoot[C]{\leavevmode
 \put(-2.5,-3){\color{myaqua}\thepage}}

 \renewcommand{\headrule}{\hbox to\headwidth{\color{myaqua}\leaders\hrule height \headrulewidth\hfill}}
\section*{Introduction}
 In this paper, we study stochastic processes of the class $(\Sigma)$. They are semi-martingales $(X_{t})_{t\geq0}$ of the form 
\begin{equation}\label{equ1}
	X_{t}=M_{t}+V_{t},
\end{equation}
where the finite variation part $V$, is an adapted continuous process such that $dV_{t}$ is carried by  the set $\{t\geq0: X_{t}=0\}$. Such processes have played an important role in many probabilistic studies. For instance: in the study of the family of Azéma-Yor martingales, the resolution of Skorokhod's embedding problem, the study of Brownian local times. They also play a key role in the study of zeros of continuous martingales. This class was introduced by Yor in \cite{y1}. Some of main properties were further studied by Nikeghbali in collaboration with Najnudel and some others with Cheridito and Platen \cite{pat,naj,naj1,naj2,naj3,nik,mult}. For instance, Nikeghbali in \cite{nik}, gives an interesting result which permits to characterize sub-martingales of the class $(\Sigma)$. It is the following theorem:
\begin{theorem}
 Let $X=M+V$ be a local sub-martingale, where $M$ is a local martingale and $V$ is the non-decreasing part of $X$ in the Doob-Meyer decomposition. The following are equivalent:
\begin{enumerate}
	\item The local sub-martingale $X$ is of class $(\Sigma)$.
	\item For every locally bounded Borel function $f$, and $F(x)\equiv\int_{0}^{x}{f(z)dz}$, the process
	$$W_{t}^{F}(X)=F(V_{t})-f(V_{t})X_{t}$$
	is a local martingale.
\end{enumerate}
 \end{theorem}
This result is extended to all semi-martingales of the class $(\Sigma)$ in Lemma 2.3 of \cite{pat}. Recently, Eyi Obiang et al in \cite{eomt} have brought some interesting contributions in this framework. In their works, they give a new way to characterize the class $(\Sigma)$. More precisely, they prove that any stochastic process is an element of the class $(\Sigma)$ if, and only if, its absolute value is a sub-martingale of the class $(\Sigma)$.

The aim of this paper is to bring our contributions to the general framework of the above mentioned class of stochastic processes. In the first one, we provide a new way to characterize stochastic processes of the class $(\Sigma)$. A consequence of this result allows to characterize any stochastic process of the class $(\Sigma)$ with respect to absolute value of some martingale. This result is very interesting because it reciprocates  a very well-known  example of processes of the class $(\Sigma)$.  More precisely, we show that for any element $X$ of the class $(\Sigma)$, there exists a martingale $M$ such that: 
$$|X_{t}|=|M_{t}|.$$

Note that the above mentioned consequence generalizes Proposition 2.3 of Ouknine and Bouhadou \cite{siam}. We also present some new properties. The second part is dedicated to the study of processes of the class $(\Sigma)$ which vanish on the set of zeros of a given Brownian motion. We provide a general framework and methods  for dealing with such processes. For instance: we establish a characterization theorem and some interesting properties for these stochastic processes. We also prove some representation results allowing to recover some stochastic processes $X$ of the class $(\Sigma)$ vanishing on the set of zeros of a Brownian motion $B$, from its final value, $X_{\infty}$ and the last time $B$ visited the origin: $\gamma=\sup\{t\geq0: B_{t}=0\}$. More precisely, from some results of enlargement filtrations theory  \cite{jeulin80} and of the theory of zeros of continuous martingales \cite{1}, we prove the next identity:
$$X_{t}=E\left[X_{\infty}1_{\{\gamma\leq t\}}|\mathcal{F}_{t}\right].$$
Analogy results were established for sub-martingales of the class $(\Sigma)$ in \cite{pat} when $\gamma$ is the last time which vanishes $X$ instead $B$.

The paper is organized as follows: in Section 1, we give some notations and recall some useful definitions and Terminologies we shall use throughout this paper. In section 2, we establish a new characterization and new properties on stochastic processes of the class $(\Sigma)$. Finally, in Section 3, we study elements of the class $(\Sigma)$ which vanish on zeros of a given Brownian motion.

For the reader's convenience, we collect in the appendix useful results from the theory of enlargement of filtrations. We also recall some results on balayage formula in the progressive case we use in this paper.

 \noindent 
\section{Notations and Terminologies}

	We begin by recalling the definition of the class $(\Sigma)$. Throughout we fix a filtered probability space \\$(\Omega,(\F_{t})_{t\geq0},\F,\Pv)$ satisfying the usual conditions.
 \begin{defn}
 We say that a stochastic process $X$ is of class $(\Sigma)$ if it decomposes as $X=M+V$, where
\begin{enumerate}
	\item $M$ is a càdlàg local martingale,
	\item $V$ is an adapted continuous finite variation process starting at 0,
	\item $\int_{0}^{t}{1_{\{X_{u}\neq0\}}dV_{u}}=0$ for all $t\geq0$.
\end{enumerate}   
 \end{defn}
 Recall that a stochastic process $X$ is said to be of class $(D)$ if $\{X_{\tau}:\tau\hspace{0.15cm} is\hspace{0.15cm} a\hspace{0.15cm} finite\hspace{0.15cm} stopping\hspace{0.15cm} time\}$ is uniformly integrable. We shall say that $X$ is of class $(\Sigma D)$ if $X$ is of class $(\Sigma)$ and of class $(D)$.

Throughout this work, we shall always use the following notations:
\begin{itemize}
	\item $B$ denotes a given Brownian motion.
	\item $L_{t}^{0}(B)$ is the local time of $B$ at level zero.
	\item $\mathcal{Z}_{1}=\{t\leq1: B_{t}=0\}$.
	\item $\gamma=\sup\{t\leq1: B_{t}=0\}$, $\gamma_{t}=\sup\{s\leq t: B_{s}=0\}$.
	\item For any process $X$, we shall denote $g=\sup\{t\geq0: X_{t}=0\}$, $g_{t}=\sup\{s\leq t: X_{s}=0\}$ and the process $K$ will be defined as:
	$$K_{t}=\lim_{s\searrow t}\inf{(1_{X_{s}>0}-1_{X_{s}<0})}.$$
	\end{itemize}
	
Now, we recall an interesting observation of Prokaj \cite{prok} which shall play a key role in this paper: For any continuous semimartingale $Y$, the set $\mathcal{W}=\{t\geq0; Y_{t}=0\}$ cannot be ordered. However, the set $\R_{+}\setminus\mathcal{W}$ can be decomposed as a countable union $\cup_{n\N}{J_{n}}$ of intervals $J_{n}$. Each interval $J_{n}$ corresponds to some excursion of $Y$. That is if $J_{n}=]g_{n},d_{n}[$, $Y_{t}\neq0$ for all $t\in]g_{n},d_{n}[$ and $Y_{g_{n}}=Y_{d_{n}}=0$. At each $J_{n}$ we associate a Bernoulli random variable $\zeta_{n}$ which is independent from any other random variables and such that 
$$P(\zeta_{n}=1)=\alpha\text{ and }P(\zeta_{n}=-1)=1-\alpha.$$
Now, let us define the process $Z^{\alpha}$ we use throughout this paper.
\begin{equation}\label{zalpha}
	Z^{\alpha}_{t}=\sum_{n=0}^{+\infty}{\zeta_{n}1_{]g_{n},d_{n}[}(t)}.
\end{equation}

We close this section by recalling an important theorem of \cite{eomt}. That is a result which gives a way to characterize stochastic processes of the class $(\Sigma)$. 
 \begin{theorem}\label{abs}
 Let $X$ be a continuous process which vanishes at zero. Then,
 $$X\in(\Sigma) \Leftrightarrow |X|\in(\Sigma).$$
 \end{theorem}

\section{New characterizations of the class \texorpdfstring{$(\Sigma)$}{sigma}}

In this section, we shall state new results to characterize stochastic processes of the class $(\Sigma)$. 
The following theorem is the main result of this paper. It gives a new way to characterize stochastic processes of the class $(\Sigma)$.
\begin{theorem}\label{z}
Let $X$ be a continuous semimartingale. The following are equivalent:
\begin{enumerate}
	\item $X\in(\Sigma)$.
	\item $\forall\alpha\in[0,1]$, $Z^{\alpha}X\in(\Sigma)$.
	\item $\exists\alpha\in[0,1]$ such that $Z^{\alpha}X\in(\Sigma)$.
\end{enumerate}
\end{theorem}
\begin{proof}
$1\Rightarrow2)$ Let $X=M+V$ be an element of the class $(\Sigma)$. One has from Proposition 2.2 of \cite{siam} what follows
$$Z^{\alpha}_{t}X_{t}=\int_{0}^{t}{Z^{\alpha}_{s}dX_{s}}+(2\alpha-1)L_{t}^{0}(Z^{\alpha}X)$$
$$\hspace{3cm}=\int_{0}^{t}{Z^{\alpha}_{s}dM_{s}}+\int_{0}^{t}{Z^{\alpha}_{s}dV_{s}}+(2\alpha-1)L_{t}^{0}(Z^{\alpha}X).$$
But, we know that $\int_{0}^{t}{Z^{\alpha}_{s}dV_{s}}=0$ since $dV_{t}$ is carried by $\{t\geq0; X_{t}=0\}$ and $X_{t}=0$ $\Leftrightarrow$ $Z^{\alpha}_{t}=0$. Hence,
\begin{equation}\label{dem}
Z^{\alpha}_{t}X_{t}=\int_{0}^{t}{Z^{\alpha}_{s}dM_{s}}+(2\alpha-1)L_{t}^{0}(Z^{\alpha}X).	
\end{equation}
Then, $Z^{\alpha}X\in(\Sigma)$ since $(2\alpha-1)dL_{t}^{0}(Z^{\alpha}X)$ is carried by $\{t\geq0; Z^{\alpha}_{t}X_{t}=0\}$ and $\left(\int_{0}^{t}{Z^{\alpha}_{s}dM_{s}};t\geq0\right)$ is a local martingale.\\
$2\Rightarrow3)$ If we consider that $\forall\alpha\in[0,1]$, $Z^{\alpha}X\in(\Sigma)$. It follows in particular that $\exists\alpha\in[0,1]$ such that $Z^{\alpha}X\in(\Sigma)$.\\
$3\Rightarrow1)$ Now, assume that $\exists\alpha\in[0,1]$ such that $Z^{\alpha}X\in(\Sigma)$. Then, according to Theorem \ref{abs}, $|Z^{\alpha}X|\in(\Sigma)$. But $\forall t\geq0$, $Z^{\alpha}_{t}\in\{-1,0,1\}$ and $Z^{\alpha}_{t}=0\Leftrightarrow X_{t}=0$. Therefore,  
$$|Z^{\alpha}X|=|X|.$$
Consequently, 
$$|X|\in(\Sigma).$$
This completes the proof.
\end{proof}

\begin{rem}\label{o}
We have proved in the above theorem that when $X\in(\Sigma)$, one has $\forall\alpha\in[0,1]$,
$$Z^{\alpha}_{t}X_{t}=\int_{0}^{t}{Z^{\alpha}_{s}dM_{s}}+(2\alpha-1)L_{t}^{0}(Z^{\alpha}X).$$
\end{rem}

Now, as an application of Theorem \ref{z}, we have the following corollary. It gives a new characterization martingale of the class $(\Sigma)$.
\begin{coro}\label{cmart}
Let $X$ be a continuous semimartingale. Then, the following holds:
$$X\in(\Sigma)\Leftrightarrow \exists\alpha\in[0,1] \text{ such that }Z^{\alpha}X \text{ is a local martingale}.$$
\end{coro}
\begin{proof}
$\Rightarrow)$ It follows from Remark \ref{o} that $\forall \alpha\in[0,1]$,
$$Z^{\alpha}_{t}X_{t}=\int_{0}^{t}{Z^{\alpha}_{s}dM_{s}}+(2\alpha-1)L_{t}^{0}(Z^{\alpha}X).$$
Hence, we obtain in particular for $\alpha=\frac{1}{2}$ that
$$Z^{\alpha}_{t}X_{t}=\int_{0}^{t}{Z^{\alpha}_{s}dM_{s}}.$$
 Therefore, $Z^{\alpha}X$ is a local martingale.

$\Leftarrow)$ Now, if we assume that $\exists\alpha\in[0,1]$ such that $Z^{\alpha}X$ is a local martingale. It follows that $Z^{\alpha}X\in(\Sigma)$. Then, it follows from Theorem \ref{z} that $X\in(\Sigma)$.
\end{proof}

It is well know that the absolute value $|M|$ of a continuous local martingale $M$, is an element of the class $(\Sigma)$. In next proposition, we show that for any stochastic process $X$ of the class $(\Sigma)$, there exists a local martingale $M$ which has the same absolute value that $X$.

\begin{prop}\label{pabs}
Let $X$ be a continuous semimartingale. Then, $X$ is an element of the class $(\Sigma)$ if and only if there exists a local martingale $M$ such that
$$|X|=|M|.$$
\end{prop}
\begin{proof}
$\Rightarrow)$ Assume that $X$ is an element of the class $(\Sigma)$ and define $Z^{\alpha}$ with $\alpha=\frac{1}{2}$. Hence, it entails from Corollary \ref{cmart} that $M=Z^{\alpha}X$ is a continuous local martingale. Then, $|X|=|M|$ since $|Z^{\alpha}X|=|X|$.

$\Leftarrow)$ Now, assume that there exists a continuous martingale $M$ such that $|X|=|M|$. We obtain from Tanaka's formula what follows
$$|X_{t}|=|M_{t}|=\int_{0}^{t}{sign(M_{s})dM_{s}}+L_{t}^{0}(M).$$
But, $L_{t}^{0}(M)=L_{t}^{0}(X)$ and $dL_{t}^{0}(X)$ is carried by $\{t\geq0: X_{t}=0\}$. Thus, $|X|\in(\Sigma)$. Consequently, it follows from Theorem \ref{abs} that $X\in(\Sigma)$. 
\end{proof}
Now, we give a series of interesting corollaries of Proposition \ref{pabs}. We begin by the next corollary which has been established by Bouhadou and Ouknine in Proposition 2.3 of \cite{siam}. 

\begin{coro}\label{sia}
If $X$ is a non-negative and continuous sub-martingale of the class $(\Sigma)$, then there exists a local martingale $M$ such that $X=|M|$.
\end{coro}

\begin{coro}\label{mart}
Let $X$ be a continuous stochastic process of the class $(\Sigma)$ and  set 
$$K_{t}=\lim_{s\searrow t}\inf{(1_{X_{s}>0}-1_{X_{s}<0})}.$$
Then, there exists a local martingale $M$ such that $\forall t\geq0$, $$X_{t}=K_{g_{t}}|M_{t}|.$$
\end{coro}
\begin{proof}
According to Proposition \ref{pabs}, there exists a local martingale $M$ such that
$$|X|=|M|.$$
This entails that 
$$K_{g_{t}}|X_{t}|=K_{g_{t}}|M_{t}|.$$
Consequently,
$$X_{t}=K_{g_{t}}|M_{t}|,$$
 Since $X_{t}=K_{g_{t}}|X_{t}|$. This completes the proof.
\end{proof}

\begin{coro}\label{p3}
Let $X$ be a process of the class $(\Sigma)$. Then, there exists $\alpha\in[0,1]$ and a process of bounded variation, adapted, continuous $R$ such that the measure $dR_{t}$ is carried by $\{t\geq0: X_{t}=0\}$ and the process 
$$\left(X_{t}-\int_{0}^{t}{^{p}(K_{s})dL^{0}_{s}(Z^{\alpha}X)}-R_{t}: t\geq0\right)$$ 
is a local martingale. 
\end{coro}
\begin{proof}
Let $X$ be a process of the class $(\Sigma)$ and take $\alpha=\frac{1}{2}$. Let us define 
$$K_{t}=\lim_{s\searrow t}\inf{(1_{X_{s}>0}-1_{X_{s}<0})}.$$
It follows from Corollary \ref{cmart} and Corollary \ref{mart} that  $M=Z^{\alpha}X$ is a local martingale and that 
$$X_{t}=K_{g_{t}}|M_{t}|.$$
Thus, according to balayage formula in the progressive case, it entails
$$X_{t}=\int_{0}^{t}{^{p}(K_{g_{s}})d|M_{s}|}+R_{t},$$
where, $R$ is a process of bounded variation, adapted, continuous such that the measure $dR_{t}$ is carried by $\{t\geq0: X_{t}=0\}$.
An application of Tanaka formula implies that
$$X_{t}=\int_{0}^{t}{^{p}(K_{g_{s}})sign(M_{s})dM_{s}}+\int_{0}^{t}{^{p}(K_{g_{s}})dL^{0}_{s}(M)}+R_{t}.$$
Therefore,
$$X_{t}-\int_{0}^{t}{^{p}(K_{g_{s}})dL^{0}_{s}(M)}-R_{t}=\int_{0}^{t}{^{p}(K_{g_{s}})sign(M_{s})dM_{s}}.$$
Consequently, $\left(X_{t}-\int_{0}^{t}{^{p}(K_{s})dL^{0}_{s}(M)}-R_{t}: t\geq0\right)$ is a local martingale.
\end{proof}

\begin{prop}\label{p4}
Let $X$ be a continuous process of the class $(\Sigma)$. Then, there exists $\alpha\in[0,1]$ such that the process 
$$\left(|X_{t}|-L^{0}_{t}(Z^{\alpha}X): t\geq0\right)$$ 
is a  continuous local martingale. 
\end{prop}

\begin{proof}
Let us take $\alpha=\frac{1}{2}$. It follows from Proposition \ref{pabs} and Corollary \ref{cmart} that $M=Z^{\alpha}X$ is a local martingale  and $\forall t\geq0$,
$$|X_{t}|=|M_{t}|.$$
Hence, thanks to Tanaka formula, we have: 
$$|X_{t}|=\int_{0}^{t}{sign(M_{s})dM_{s}}+L^{0}_{t}(M).$$
Therefore,
$$|X_{t}|-L^{0}_{t}(M)=\int_{0}^{t}{sign(M_{s})dM_{s}}$$
is a local martingale.
\end{proof}
\begin{coro}
Let $X$ be a non-negative process of the class $(\Sigma)$. Then, the process 
$$\left(X_{t}-L^{0}_{t}(Z^{\alpha}X): t\geq0\right)$$ 
is a local martingale. 
\end{coro}
\begin{proof}
That is a direct consequence of Proposition \ref{p4}, since $|X_{t}|=X_{t}$.
\end{proof}





\section{Processes of the class \texorpdfstring{$(\Sigma)$}{sigma} which vanish on zeros of a Brownian motion}
{

Now, we shall apply results of the previous section to describe $(\Sigma_{B}^{0})$, the set of processes $(X_{t}: t\geq0)$ of the class $(\Sigma)$ which vanish on the zero set of a given Brownian motion $(B_{t}: t\geq0)$. 
\subsection{Characterization and properties}
\begin{defn}

Given a Brownian motion $B$, we denote by $(\Sigma_{B}^{0})$, the set of processes of the class $(\Sigma)$  which vanish on the zero set of $B$ and by $(\Sigma_{B,+}^{0})$, the set of non-negative processes of $(\Sigma_{B}^{0})$.
\end{defn}

In what follows, we give a characterization of processes of the class $(\Sigma_{B}^{0})$.
\begin{theorem}\label{th1}
$X\in$$\Sigma_{B}^{0}$ if, and only if, $X$ may be written as:
\begin{equation}
	X_{t}=K_{g_{t}}z_{\gamma_{t}}|B_{t}|\exp\left(\int_{\gamma_{t}}^{t}{u_{s}\left(dB_{s}-\frac{ds}{B_{s}}\right)-\frac{1}{2}\int_{\gamma_{t}}^{t}{u_{s}^{2}ds}}\right)
\end{equation}
where $(z_{t};t\geq0)$ and $(u_{t};t\geq0)$ are two predictable processes with respect to $(\mathcal{F}_{t})_{t\geq0}$ with suitable integrability properties. $(z_{\gamma_{t}};t\geq0)$ may be obtained as 
$$z_{\gamma_{t}}=\lim_{u\to\gamma_{t}}{\frac{|X_{u}|}{|B_{u}|}}.$$
Hence,
\begin{equation}
	|X_{t}|=z_{\gamma_{t}}|B_{t}|\exp\left(\int_{\gamma_{t}}^{t}{u_{s}\left(dB_{s}-\frac{ds}{B_{s}}\right)-\frac{1}{2}\int_{\gamma_{t}}^{t}{u_{s}^{2}ds}}\right).
\end{equation}
\end{theorem}
\begin{proof}

We have from Theorem \ref{abs} that $|X|\in(\Sigma)$. Then, it follows thanks to Proposition \ref{pabs} there exists a local martingale $M$ such that $\forall t\geq0$,
$$|X_{t}|=|M_{t}|.$$
$M$ vanishes on $\mathcal{Z}_{1}$ since $X\in$$\Sigma_{B}^{0}$. Hence, from Theorem 3.3 of \cite{man}, $M$ may be written as 
$$M_{t}=z^{'}_{\gamma_{t}}B_{t}\exp\left(\int_{\gamma_{t}}^{t}{u_{s}\left(dB_{s}-\frac{ds}{B_{s}}\right)-\frac{1}{2}\int_{\gamma_{t}}^{t}{u_{s}^{2}ds}}\right),$$
with ${z^{'}_{\gamma_{t}}=\lim_{u\to\gamma_{t}}{\frac{M_{u}}{B_{u}}}}.$ Therefore,
$$|X_{t}|=|z^{'}_{\gamma_{t}}||B_{t}|\exp\left(\int_{\gamma_{t}}^{t}{u_{s}\left(dB_{s}-\frac{ds}{B_{s}}\right)-\frac{1}{2}\int_{\gamma_{t}}^{t}{u_{s}^{2}ds}}\right).$$
It follows by multiplying by $K_{g_{t}}$ that:
$$K_{g_{t}}|X_{t}|=K_{g_{t}}|z^{'}_{\gamma_{t}}||B_{t}|\exp\left(\int_{\gamma_{t}}^{t}{u_{s}\left(dB_{s}-\frac{ds}{B_{s}}\right)-\frac{1}{2}\int_{\gamma_{t}}^{t}{u_{s}^{2}ds}}\right).$$
Consequently,
$$X_{t}=K_{g_{t}}z_{\gamma_{t}}|B_{t}|\exp\left(\int_{\gamma_{t}}^{t}{u_{s}\left(dB_{s}-\frac{ds}{B_{s}}\right)-\frac{1}{2}\int_{\gamma_{t}}^{t}{u_{s}^{2}ds}}\right),$$
since $X_{t}=K_{g_{t}}|X_{t}|$ and by continuity of the map $x\mapsto|x|$, $z_{\gamma_{t}}=|z^{'}_{\gamma_{t}}|$.
\end{proof}

In what follows, we give a corollary of the above theorem.
\begin{coro}\label{c4}
If $X\in(\Sigma_{B,+}^{0})$, then, $X$ may be written as:
\begin{equation}
	X_{t}=z_{\gamma_{t}}|B_{t}|\exp\left(\int_{\gamma_{t}}^{t}{u_{s}\left(dB_{s}-\frac{ds}{B_{s}}\right)-\frac{1}{2}\int_{\gamma_{t}}^{t}{u_{s}^{2}ds}}\right).
\end{equation}
\end{coro}
\begin{proof}
It is enough to see that  $K_{g_{t}}=1$ since $X$ is a non-negative process.
\end{proof}

\begin{proof}
That is a direct application of Theorem \ref{abs} and Corollary \ref{c4}. 
\end{proof}

\begin{prop}\label{pm}
Let $X$ be a continuous process of the class $(\Sigma)$. If $X$ vanishes on $\mathcal{Z}=\{t\leq1: B_{t}=0\}$, then $$\left(|X_{t}|-\int_{0}^{t}{z_{s}dL_{s}^{0}(B)};t\geq0\right)$$ is a local martingale.
\end{prop}
\begin{proof}
We have from Corollary \ref{th1}, there exist two bounded predictable processes $z$ and $u$ such:
$$|X_{t}|=z_{\gamma_{t}}|B_{t}|\exp\left(\int_{\gamma_{t}}^{t}{u_{s}\left(dB_{s}-\frac{ds}{B_{s}}\right)-\frac{1}{2}\int_{\gamma_{t}}^{t}{u_{s}^{2}ds}}\right).$$
Let $Y_{t}=z_{\gamma_{t}}|B_{t}|$ and $$W_{t}=\int_{\gamma_{t}}^{t}{u_{s}\left(dB_{s}-\frac{ds}{B_{s}}\right)-\frac{1}{2}\int_{\gamma_{t}}^{t}{u_{s}^{2}ds}}=\int_{0}^{t}{u_{s}1_{]\gamma_{t};t]}\left(dB_{s}-\frac{ds}{B_{s}}\right)-\frac{1}{2}\int_{0}^{t}{u_{s}^{2}1_{]\gamma_{t};t]}ds}}.$$
An application of the balayage formula implies what follows
$$z_{\gamma_{t}}|B_{t}|=\int_{0}^{t}{z_{\gamma_{s}}d|B_{s}|}=\int_{0}^{t}{z_{\gamma_{s}}{\rm sgn}(B_{s})dB_{s}}+\int_{0}^{t}{z_{\gamma_{s}}dL_{s}^{0}(B)}.$$
But, $\gamma_{s}=s$, $dL_{s}^{0}(B)-$ almost surely. Thus,
$$Y_{t}=z_{\gamma_{t}}|B_{t}|=\int_{0}^{t}{z_{\gamma_{s}}{\rm sgn}(B_{s})dB_{s}}+\int_{0}^{t}{z_{s}dL_{s}^{0}(B)}.$$
Now, we shall apply the Ito formula on $|X|=Ye^{W}$. One has $\forall t\geq0$,
\begin{equation}\label{e1}
	|X_{t}|=\int_{0}^{t}{Y_{s}e^{W_{s}}dW_{s}}+\int_{0}^{t}{e^{W_{s}}dY_{s}}+\frac{1}{2}\int_{0}^{t}{e^{W_{s}}d\langle W,Y\rangle_{s}}+\frac{1}{2}\int_{0}^{t}{e^{W_{s}}d\langle W,Y\rangle_{s}}+\frac{1}{2}\int_{0}^{t}{Y_{s}e^{W_{s}}d\langle W\rangle_{s}}.
\end{equation}
Hence,
$$|X_{t}|=\int_{0}^{t}{Y_{s}e^{W_{s}}dW_{s}}+\int_{0}^{t}{e^{W_{s}}dY_{s}}+\int_{0}^{t}{e^{W_{s}}d\langle W,Y\rangle_{s}}+\frac{1}{2}\int_{0}^{t}{Y_{s}e^{W_{s}}d\langle W\rangle_{s}}.$$
But remark that:
$$\int_{0}^{t}{Y_{s}e^{W_{s}}dW_{s}}=\int_{\gamma_{t}}^{t}{Y_{s}e^{W_{s}}\left[u_{s}\left(dB_{s}-\frac{ds}{B_{s}}\right)-\frac{1}{2}u_{s}^{2}ds\right]}.$$
We have more precisely,
$$\int_{0}^{t}{Y_{s}e^{W_{s}}dW_{s}}=\int_{\gamma_{t}}^{t}{z_{\gamma_{s}}|B_{s}|e^{W_{s}}u_{s}dB_{s}}-\int_{\gamma_{t}}^{t}{u_{s}z_{\gamma_{s}}\frac{|B_{s}|}{B_{s}}e^{W_{s}}ds}-\frac{1}{2}\int_{\gamma_{t}}^{t}{z_{\gamma_{s}}|B_{s}|u_{s}^{2}e^{W_{s}}ds}$$
$$\hspace{3cm}=\int_{\gamma_{t}}^{t}{z_{\gamma_{s}}|B_{s}|e^{W_{s}}u_{s}dB_{s}}-\int_{\gamma_{t}}^{t}{u_{s}z_{\gamma_{s}}{\rm sgn}(B_{s})e^{W_{s}}ds}-\frac{1}{2}\int_{\gamma_{t}}^{t}{z_{\gamma_{s}}|B_{s}|u_{s}^{2}e^{W_{s}}ds}.$$ 
On the other hand, it is easy to verify that:
$$\int_{0}^{t}{e^{W_{s}}d\langle W,Y\rangle_{s}}+\frac{1}{2}\int_{0}^{t}{Y_{s}e^{W_{s}}d\langle W\rangle_{s}}=\int_{\gamma_{t}}^{t}{u_{s}z_{\gamma_{s}}{\rm sgn}(B_{s})e^{W_{s}}ds}+\frac{1}{2}\int_{\gamma_{t}}^{t}{z_{\gamma_{s}}|B_{s}|u_{s}^{2}e^{W_{s}}ds}.$$ 
Then, we obtain from equation \eqref{e1}
$$|X_{t}|=\int_{\gamma_{t}}^{t}{z_{\gamma_{s}}|B_{s}|e^{W_{s}}u_{s}dB_{s}}+\int_{0}^{t}{e^{W_{s}}dY_{s}}$$
$$\hspace{5.5cm}=\int_{\gamma_{t}}^{t}{z_{\gamma_{s}}|B_{s}|e^{W_{s}}u_{s}dB_{s}}+\int_{0}^{t}{e^{W_{s}}z_{\gamma_{s}}{\rm sgn}(B_{s})dB_{s}}+\int_{0}^{t}{e^{W_{s}}z_{s}dL^{0}_{s}(B)}.$$
Recall that $\gamma_{s}=s$, $dL^{0}_{s}(B)$ - almost surely. That implies that $e^{W_{s}}=e^{W_{\gamma_{s}}}=1$, $dL^{0}_{s}(B)$ - almost surely. Consequently,
$$|X_{t}|=\int_{\gamma_{t}}^{t}{z_{\gamma_{s}}|B_{s}|e^{W_{s}}u_{s}dB_{s}}+\int_{0}^{t}{e^{W_{s}}dY_{s}}$$
$$\hspace{5.25cm}=\int_{\gamma_{t}}^{t}{z_{\gamma_{s}}|B_{s}|e^{W_{s}}u_{s}dB_{s}}+\int_{0}^{t}{e^{W_{s}}z_{\gamma_{s}}{\rm sgn}(B_{s})dB_{s}}+\int_{0}^{t}{z_{s}dL^{0}_{s}(B)}.$$
This implies that $\left(|X_{t}|-\int_{0}^{t}{z_{s}dL_{s}^{0}(B)};t\geq0\right)$ is a local martingale.
\end{proof}

Now, we shall give some corollaries of Proposition \ref{pm}.

\begin{coro}\label{c6}
If $X$ is a process of the class $(\Sigma)$ which vanishes on $\mathcal{Z}_{1}$. Then the Doob decomposition of the sub-martingale $|X|$ is given by the next formula
$$|X_{t}|=\int_{0}^{t}{z_{\gamma_{s}}{\rm sgn}(B_{s})e^{W_{s}}\left(1+B_{s}u_{s}1_{]\gamma_{t};t]}(s)\right)dB_{s}}+\int_{0}^{t}{z_{s}dL_{s}^{0}(B)},$$
with $$W_{t}=\int_{\gamma_{t}}^{t}{u_{s}\left(dB_{s}-\frac{ds}{B_{s}}\right)-\frac{1}{2}\int_{\gamma_{t}}^{t}{u_{s}^{2}ds}}.$$
\end{coro}
\begin{coro}\label{ccc6}
Let $X$ be a process of the class $(\Sigma)$ vanishing on $\mathcal{Z}_{1}$ and $(z_{t};t\geq0)$ its associated bounded predictable process. If $(z_{\gamma_{t}};t\geq0)$ vanishes on $\mathcal{Z}_{1}$. Then, $(|X_{t}|;t\geq0)$ is a local martingale.
\end{coro}
\begin{proof}
According to Corollary \ref{c6}, one has
$$|X_{t}|=\int_{0}^{t}{z_{\gamma_{s}}{\rm sgn}(B_{s})e^{W_{s}}\left(1+B_{s}u_{s}1_{]\gamma_{t};t]}(s)\right)dB_{s}}+\int_{0}^{t}{z_{s}dL_{s}^{0}(B)}.$$
But, 
$$\int_{0}^{t}{z_{s}dL_{s}^{0}(B)}=\int_{0}^{t}{z_{\gamma_{s}}dL_{s}^{0}(B)}=\int_{0}^{t}{z_{\gamma_{s}}1_{\{s\geq0;B_{s}=0\}}dL_{s}^{0}(B)}$$
because $dL_{s}^{0}(B)$ is carried by $\mathcal{Z}_{1}$. Therefore,
$$\int_{0}^{t}{z_{s}dL_{s}^{0}(B)}=0$$
since $(z_{\gamma_{t}};t\geq0)$ vanishes on $\mathcal{Z}_{1}$. Consequently, $(|X_{t}|;t\geq0)$ is a local martingale.
\end{proof}
\begin{coro}\label{cc6}
Let $X$ be a process of the class $(\Sigma)$. If $X$ vanishes on $\mathcal{Z}=\{t\leq1: B_{t}=0\}$, then $$\left(|X_{t}|-z_{\gamma_{t}}|B_{t}|;t\geq0\right)$$ is a local martingale.
\end{coro}
\begin{proof}
We can say thanks to Corollary \ref{c6} that
$$|X_{t}|=\int_{0}^{t}{z_{\gamma_{s}}{\rm sgn}(B_{s})e^{W_{s}}\left(1+B_{s}u_{s}1_{]\gamma_{t};t]}(s)\right)dB_{s}}+\int_{0}^{t}{z_{s}dL_{s}^{0}(B)}.$$
On another hand, we have from balayage formula what follows
$$z_{\gamma_{t}}|B_{t}|=\int_{0}^{t}{z_{\gamma_{s}}{\rm sgn}(B_{s})dB_{s}}+\int_{0}^{t}{z_{s}dL_{s}^{0}(B)}.$$
That implies that
$$|X_{t}|-z_{\gamma_{t}}|B_{t}|=\int_{0}^{t}{z_{\gamma_{s}}{\rm sgn}(B_{s})e^{W_{s}}\left(1+B_{s}u_{s}1_{]\gamma_{t};t]}(s)\right)dB_{s}}-\int_{0}^{t}{{z_{\gamma_{s}}{\rm sgn}(B_{s})dB_{s}}}$$
$$\hspace{-0.2cm}=\int_{0}^{t}{z_{\gamma_{s}}{\rm sgn}(B_{s})\left(e^{W_{s}}(1+B_{s}u_{s}1_{]\gamma_{t};t]}(s))-1\right)dB_{s}}.$$
\end{proof}

\begin{coro}\label{c8}
If $X$ is a process of the class $(\Sigma)$ which vanishes on $\mathcal{Z}_{1}$. Then, there exists a process $R$ of bounded variation, adapted, continuous such that the measure $dR_{t}$ is carried by $\{t\geq0: X_{t}=0\}$ and the process 
$$\left(X_{t}-\int_{0}^{t}{^{p}(K_{s})z_{s}dL^{0}_{s}(B)}-R_{t}: t\geq0\right)$$ 
is a local martingale. 
\end{coro}
\begin{proof}
We have from Corollary \ref{p3} that there exists a process $R$ of bounded variation, adapted, continuous such that the measure $dR_{t}$ is carried by $\{t\geq0: X_{t}=0\}$ such that the process 
$$\left(X_{t}-\int_{0}^{t}{^{p}(K_{s})dL^{0}_{s}(M)}-R_{t}: t\geq0\right)$$ 
is a local martingale. But, according to Proposition \ref{pabs}, $|X|=|M|$. Thus, we have from Corollary \ref{c6} that
$$dL^{0}_{s}(M)=z_{s}dL^{0}_{s}(B).$$
Therefore,
$$\left(X_{t}-\int_{0}^{t}{^{p}(K_{s})z_{s}dL^{0}_{s}(B)}-R_{t}: t\geq0\right)$$ 
is a local martingale.
\end{proof}

\begin{prop}\label{p7}
Let $X$ be a process of the class $(\Sigma_{B}^{0})$ and $f$, a locally bounded Borel function. Define
$$F(z_{t};L_{t}^{0}(B))=\int_{0}^{t}{z_{s}f(L_{s}^{0}(B))dL_{s}^{0}(B)}.$$
Then, the process $\left(f(L_{s}^{0}(B))|X_{t}|-F(z_{t};L_{t}^{0}(B)); t\geq0\right)$ is a local martingale.
\end{prop}
\begin{proof}
We have thanks to balayage formula,
$$f(L_{s}^{0}(B))|X_{t}|=\int_{0}^{t}{f(L_{s}^{0}(B))d|X_{s}|}.$$
It follows from Corollary \ref{c6} that
$$f(L_{t}^{0}(B))|X_{t}|=\int_{0}^{t}{f(L_{s}^{0}(B))z_{\gamma_{s}}{\rm sgn}(B_{s})e^{W_{s}}\left(1+B_{s}u_{s}1_{]\gamma_{t};t]}(s)\right)dB_{s}}+\int_{0}^{t}{f(L_{s}^{0}(B))z_{s}dL_{s}^{0}(B)}.$$
Then,
$$f(L_{t}^{0}(B))|X_{t}|-F(Z_{t};L_{t}^{0}(B))=\int_{0}^{t}{f(L_{s}^{0}(B))z_{\gamma_{s}}{\rm sgn}(B_{s})e^{W_{s}}\left(1+B_{s}u_{s}1_{]\gamma_{t};t]}(s)\right)dB_{s}}.$$
Consequently, the theorem is proved..
\end{proof}
\begin{rem}\label{r1}
We have proved that for $f$ a locally bounded Borel function, $(f(L_{t}^{0}(B))|X_{t}|; t\geq0)$ is also a process of $(\Sigma_{B}^{0})$ and its finite variation process is $(F(z_{t};L_{t}^{0}(B));t\geq0)$.
\end{rem}
The next corollary gives us conditions under which $\left(f(L_{s}^{0}(B))|X_{t}|-F(z_{t};L_{t}^{0}(B)); t\geq0\right)$ is a true martingale.
\begin{coro}\label{unif}
Let $X$ be a process of class $(\Sigma_{B}^{0})$ such that $|X|$ is of class $(D)$. If $f$ is a Borel bounded function with compact support, then $\left(f(L_{s}^{0}(B))|X_{t}|-F(z_{t};L_{t}^{0}(B)); t\geq0\right)$ is a uniformly integrable martingale.
\end{coro}
\begin{proof}
There exist two constants $C>0$ and $K>0$ such that $\forall x\geq0$, $|f(x)|\leq C$ and $\forall x\geq K$, $f(x)=0$. We have:
$$|F(z_{t};L_{t}^{0}(B))-f(L_{s}^{0}(B))|X_{t}||=\left|\int_{0}^{t}{f(L_{s}^{0}(B))z_{s}dL_{s}^{0}(B)}-f(L_{s}^{0}(B))|X_{t}|\right|$$
We know that $z$ is a non-negative bounded process. Then, there exists a constant $\lambda>0$ such that $\forall t\geq0$, $z_{t}\leq\lambda$. Therefore,
$$|F(z_{t};L_{t}^{0}(B))-f(L_{s}^{0}(B))|X_{t}||\leq \lambda \left|\int_{0}^{t}{f(L_{s}^{0}(B))dL_{s}^{0}(B)}\right|+|f(L_{s}^{0}(B))||X_{t}|.$$
Consequently, we have
$$|F(z_{t};L_{t}^{0}(B))-f(L_{s}^{0}(B))|X_{t}||\leq\lambda KC+C|X_{t}|.$$
We deduce that $\left(f(L_{s}^{0}(B))|X_{t}|-F(z_{t};L_{t}^{0}(B)); t\geq0\right)$ is a local martingale of class $(D)$ since \\$(\lambda KC+C|X_{t}|;t\geq0)$ is of class $(D)$. This achieves the proof.
\end{proof}
\begin{coro}\label{c10}
Let $X$ be a process of the class $(\Sigma_{B}^{0})$ and $f$ be a locally bounded Borel function. 
Then, there exists a process of bounded variation, adapted and continuous $R$ such that the measure $dR_{t}$ is carried by $\{t\geq0;f(L_{t}^{0}(B)X_{t}=0\}$ and the process $$\left(f(L_{s}^{0}(B))X_{t}-\int_{0}^{t}{^{p}(K_{s})f(L_{s}^{0}(B))z_{s}dL_{s}^{0}(B)}-R_{t}; t\geq0\right)$$ is a local martingale.
\end{coro}
\begin{proof}
We know that for every $t\geq0$, $X_{t}=K_{g_{t}}|X_{t}|$. If we note $Y_{t}=f(L_{t}^{0}(B))|X_{t}|$, it follows from balayage formula that
$$f(L_{t}^{0}(B))X_{t}=K_{g_{t}}Y_{t}=\int_{0}^{t}{^{p}(K_{g_{s}})dY_{s}} +R_{t},$$
where $R$ is a process of bounded variation, adapted and continuous such that the measure $dR_{t}$ is carried by $\{t\geq0;f(L_{t}^{0}(B)X_{t}=0\}$.
An application of Remark \ref{r1} implies that
$$f(L_{t}^{0}(B))X_{t}=\int_{0}^{t}{^{p}(K_{g_{s}})f(L_{s}^{0}(B))z_{\gamma_{s}}{\rm sgn}(B_{s})e^{W_{s}}\left(1+B_{s}u_{s}1_{]\gamma_{t};t]}(s)\right)dB_{s}}+\int_{0}^{t}{^{p}(K_{g_{s}})f(L_{s}^{0}(B))z_{s}dL_{s}^{0}(B)}+R_{t}.$$
Since we have ($z_{s}dL_{s}^{0}(B))$- almost surely, $g_{s}=s$.  Therefore, it follows that
$$f(L_{t}^{0}(B))X_{t}=\int_{0}^{t}{^{p}(K_{g_{s}})f(L_{s}^{0}(B))z_{\gamma_{s}}{\rm sgn}(B_{s})e^{W_{s}}\left(1+B_{s}u_{s}1_{]\gamma_{t};t]}(s)\right)dB_{s}}+\int_{0}^{t}{^{p}(K_{s})f(L_{s}^{0}(B))z_{s}dL_{s}^{0}(B)}+R_{t}.$$
Consequently, $$\left(f(L_{s}^{0}(B))X_{t}-\int_{0}^{t}{^{p}(K_{s})f(L_{s}^{0}(B))z_{s}dL_{s}^{0}(B)}-R_{t}; t\geq0\right)$$ is a local martingale.
\end{proof}
\begin{rem}
The process $$\left(f(L_{s}^{0}(B))X_{t}-\int_{0}^{t}{^{p}(K_{s})f(L_{s}^{0}(B))z_{s}dL_{s}^{0}(B)}-R_{t}; t\geq0\right)$$ is a uniformly integrable martingale under assumptions of Corollary \ref{unif}.
\end{rem}

Mansury and Yor in \cite{man} have characterized $\mathcal{M}_{0}^{strict}$, the set of Brownian martingales whose zero set coincides exactly with that Brownian motion as follows:
\begin{theorem}\label{tman}
A martingale $(M_{t};t\geq0)$ belongs to $\mathcal{M}_{0}^{strict}$ if, and only if $M$ may be written 
$$z_{\gamma_{t}}B_{t}\exp\left(\int_{\gamma_{t}}^{t}{u_{s}\left(dB_{s}-\frac{ds}{B_{s}}\right)-\frac{1}{2}\int_{\gamma_{t}}^{t}{u_{s}^{2}ds}}\right)$$
with $P(\exists u\geq0, z_{\gamma_{u}}=0; u\neq \gamma_{u})=0$.
\end{theorem}

We note that, from Theorem \ref{tman} and Theorem \ref{th1} , a process $X$ belongs to $(\Sigma)$ with $\{t\geq0; X_{t}=0\}=\{t\geq0; B_{t}=0\}$ if, and only if $|X|$ may be written
$$z_{\gamma_{t}}|B_{t}|\exp\left(\int_{\gamma_{t}}^{t}{u_{s}\left(dB_{s}-\frac{ds}{B_{s}}\right)-\frac{1}{2}\int_{\gamma_{t}}^{t}{u_{s}^{2}ds}}\right)$$
with $P(\exists u\geq0, z_{\gamma_{u}}=0; u\neq \gamma_{u})=0$.

In next theorem, we shall characterize processes of the class $(\Sigma)$ whose zero set coincides exactly with that of Brownian motion.
\begin{theorem}
Let $X$ be a process of the class $(\Sigma)$. The following holds.

$\{t\geq0:X_{t}=0\}=\{t\geq0:B_{t}=0\}$ if, and only if, there exists a non-negative and bounded predictable process $z^{'}$ satisfying $P(\exists u\geq0, z^{'}_{\gamma_{u}}=0; u\neq \gamma_{u})=0$ such that $\left(|X_{t}|-\int_{0}^{t}{z^{'}_{s}dL_{s}^{0}(B)};t\geq0\right)$ is a local martingale.
\end{theorem}
\begin{proof}
$\Rightarrow)$ Let us take $z^{'}=z$. Since $\{t\geq0:X_{t}=0\}=\{t\geq0:B_{t}=0\}$. We have by definition that $$P(\exists u\geq0, z_{\gamma_{u}}=0; u\neq \gamma_{u})=0.$$ Furthermore, we obtain thanks to Proposition \ref{pm} that $$\left(|X_{t}|-\int_{0}^{t}{z_{s}dL_{s}^{0}(B)};t\geq0\right)$$ is a local martingale.\\
$\Leftarrow)$ Now, assume that there exists a non-negative and bounded predictable process $z^{'}$ such that $$P(\exists u\geq0, z^{'}_{\gamma_{u}}=0; u\neq \gamma_{u})=0$$ and $$\left(|X_{t}|-\int_{0}^{t}{z^{'}_{s}dL_{s}^{0}(B)};t\geq0\right)$$ is a local martingale. Then, 
$$A_{t}=\int_{0}^{t}{z^{'}_{s}dL_{s}^{0}(B)}$$
is the non-decreasing process of the sub-martingale $|X|$. We know that $\{t\geq0:X_{t}=0\}$ is the support of $dA_{t}$ since $|X|$ belongs to $(\Sigma)$. Furthermore, $\{t\geq0:B_{t}=0\}$ is the support of $dL_{t}^{0}(B)$. Consequently, 
$$\{t\geq0:X_{t}=0\}=\{t\geq0:B_{t}=0\}$$
since $dA_{t}=z^{'}_{t}dL_{t}^{0}(B)$ and $P(\exists u\geq0, z^{'}_{\gamma_{u}}=0; u\neq \gamma_{u})=0$.
\end{proof}

\begin{prop}\label{p10}
Let $X$ be a process of the class $(\Sigma)$ which vanishes on $\{t\geq0: B_{t}=0\}$. Let $z$ its associated predictable process. If for every $t\geq0$, $z_{\gamma_{t}}=1$. Then, $\{t\geq0: X_{t}=0\}=\{t\geq0: B_{t}=0\}$ and $L_{t}^{0}(B)$ is the increasing part of the sub-martingale $|X|$.
\end{prop}
\begin{proof}
By applying Corollary \ref{th1}, we obtain that
$$|X_{t}|=|B_{t}|\exp\left(\int_{\gamma_{t}}^{t}{u_{s}\left(dB_{s}-\frac{ds}{B_{s}}\right)}-\frac{1}{2}\int_{\gamma_{t}}^{t}{u_{s}^{2}ds}\right).$$
That implies, $\{t\geq0: X_{t}=0\}=\{t\geq0: B_{t}=0\}$. Furthermore, we have from Proposition \ref{pm} that $(|X_{t}|-L_{t}^{0}(B):t\geq0)$ is a local martingale. This closes proof.
\end{proof}
\begin{coro}\label{c11}
Let $X$ be a process of the class $(\Sigma)$ whose zero set coincides exactly with that of Brownian motion and $z$ its associated bounded predictable process. Then, $(\frac{1}{z_{\gamma_{t}}}|X_{t}|;t\geq0)$ is also a sub-martingale of the class $(\Sigma)$ whose zero set coincides exactly with that of Brownian motion  and its increasing part is $(L_{t}^{0}(B); t\geq0)$.
\end{coro}
\begin{proof}
According to Theorem \ref{th1}, we have
$$|X_{t}|=|z_{\gamma_{t}}||B_{t}|\exp\left(\int_{\gamma_{t}}^{t}{u_{s}\left(dB_{s}-\frac{ds}{B_{s}}\right)-\frac{1}{2}\int_{\gamma_{t}}^{t}{u_{s}^{2}ds}}\right).$$
Hence, it follows that
$$\frac{1}{|z_{\gamma_{t}}|}|X_{t}|=|B_{t}|\exp\left(\int_{\gamma_{t}}^{t}{u_{s}\left(dB_{s}-\frac{ds}{B_{s}}\right)-\frac{1}{2}\int_{\gamma_{t}}^{t}{u_{s}^{2}ds}}\right).$$
Therefore, $(\frac{1}{z_{\gamma_{t}}}|X_{t}|;t\geq0)$ is an element of the class $(\Sigma)$ vanishing on $\{t\geq0: B_{t}=0\}$ and its associated predictable process is $z^{'}\equiv1$. Consequently, we can conclude this proof by applying Proposition \ref{p10}.
\end{proof}

We give in what follows, another result which allows to characterize processes of the class $(\Sigma_{B}^{0,strict})$.
\begin{theorem}\label{t11}
Let $X$ be a process of the class $(\Sigma)$. Then, $\mathcal{Z}=\{t\leq1: X_{t}=0\}$ if, and only if, there exists a non-negative and bounded predictable process $z$ satisfying $P(\exists u\geq0, z_{\gamma_{u}}=0; u\neq \gamma_{u})=0$ such that  $$\left(|X_{t}|-z_{\gamma_{t}}|B_{t}|;t\geq0\right)$$ is a local martingale.
\end{theorem}
\begin{proof}
$\Rightarrow)$ Assume that $\mathcal{Z}=\{t\leq1: X_{t}=0\}$. We have from Corollary \ref{c6} that
$$|X_{t}|=\int_{0}^{t}{z_{\gamma_{s}}{\rm sgn}(B_{s})e^{W_{s}}\left(1+B_{s}u_{s}1_{]\gamma_{t};t]}(s)\right)dB_{s}}+\int_{0}^{t}{z_{s}dL_{s}^{0}(B)},$$
with $$W_{t}=\int_{\gamma_{t}}^{t}{u_{s}\left(dB_{s}-\frac{ds}{B_{s}}\right)-\frac{1}{2}\int_{\gamma_{t}}^{t}{u_{s}^{2}ds}}.$$
Furthermore, we obtain from an application of balayage formula that
$$z_{\gamma_{t}}|B_{t}|=\int_{0}^{t}{z_{\gamma_{s}}{\rm sgn}(B_{s})dB_{s}}+\int_{0}^{t}{z_{s}dL_{s}^{0}(B)}.$$
Hence, it follows that
$$|X_{t}|-z_{\gamma_{t}}|B_{t}|=\int_{0}^{t}{z_{\gamma_{s}}{\rm sgn}(B_{s})e^{W_{s}}\left(1+B_{s}u_{s}1_{]\gamma_{t};t]}(s)\right)dB_{s}}+\int_{0}^{t}{z_{\gamma_{s}}{\rm sgn}(B_{s})dB_{s}}.$$
Therefore, $$\left(|X_{t}|-z_{\gamma_{t}}|B_{t}|;t\geq0\right)$$ is a local martingale.\\
$\Leftarrow)$ Now, assume that there exists a non-negative and bounded predictable process $z$ satisfying $P(\exists u\geq0, z_{\gamma_{u}}=0; u\neq \gamma_{u})=0$ such that  $$\left(|X_{t}|-z_{\gamma_{t}}|B_{t}|;t\geq0\right)$$ is a local martingale. We know that $X$ is a sub-martingale of the class $(\Sigma)$. Thus, $|X|$ can be written as $|X|=M+A$, where $M$ is a local martingale and $A$ is a non-decreasing, positive and continuous process such that $dA_{t}$ is carried by $\{t\geq0: X_{t}=0\}$. Therefore, we obtain that
$$|X_{t}|-z_{\gamma_{t}}|B_{t}|=\left(M_{t}-\int_{0}^{t}{z_{\gamma_{s}}{\rm sgn}(B_{s})dB_{s}}\right)+\left(A_{t}-\int_{0}^{t}{z_{s}dL_{s}^{0}(B)}\right).$$
That implies that $\left(A_{t}-\int_{0}^{t}{z_{s}dL_{s}^{0}(B)};t\geq0\right)$ is a continuous martingale. Thus, we obtain:
$$A_{t}=\int_{0}^{t}{z_{s}dL_{s}^{0}(B)}.$$
And, $dA_{t}=z_{t}dL_{t}^{0}(B)$. Consequently, $dA_{t}$ and $dL_{t}^{0}(B)$ have the same support since $P(\exists u\geq0, z_{\gamma_{u}}=0; u\neq \gamma_{u})=0$. This completes the proof.
\end{proof}
\begin{coro}
Let $X$ be a process of the class $(\Sigma)$ such that $$\left(|X_{t}|-|B_{t}|;t\geq0\right)$$ is a local martingale. Then, $\mathcal{Z}=\{t\leq1: X_{t}=0\}$.
\end{coro}

\begin{coro}
Let $X$ be a process of the class $(\Sigma)$ vanishing on $\mathcal{Z}_{1}$. If $z$ is its associated predictable process  such that $\forall t\geq0$, $z_{\gamma_{t}}=1$. Then,  $$\left(|X_{t}|-|B_{t}|;t\geq0\right)$$ is a local martingale.
\end{coro}
\begin{theorem}\label{estim}
Let $X$ be a process of the class $(\Sigma)$ whose zero set coincides exactly with that of Brownian motion. Define $(\tau_{u})$ the right continuous inverse of $(L_{t}^{0}(B);t\geq0)$: 
$$ \tau_{u}=\inf{\{t\geq0:L_{t}^{0}(B)>u\}}.$$
Let $\varphi:\R_{+}\rightarrow\R_{+}$ be a Borel function. Then, we have the following estimates:
\begin{equation}
	P\left(\exists t\geq0, |X_{t}|>z_{\gamma_{t}}\varphi(L_{t}^{0}(B))\right)=1-\exp\left(-\int_{0}^{+\infty}{\frac{dx}{\varphi(x)}}\right)
\end{equation}
and
\begin{equation}
	P\left(\exists t\leq\tau, |X_{t}|>z_{\gamma_{t}}\varphi(L_{t}^{0}(B))\right)=1-\exp\left(-\int_{0}^{u}{\frac{dx}{\varphi(x)}}\right)
\end{equation}
\end{theorem}
\begin{proof}
According to Corollary \ref{c11}, $\left(\frac{|X_{t}|}{z_{\gamma_{t}}};t\geq0\right)$ is a sub-martingale of the class $(\Sigma)$ whose zero set coincides exactly with $\mathcal{Z}=\{t\leq1:B_{t}=0\}$ and its increasing process is $(L_{t}^{0}(B);t\geq0)$. We obtain from Theorem 3.2 of \cite{nik} that:
 \begin{equation}
	P\left(\exists t\geq0, \frac{|X_{t}|}{z_{\gamma_{t}}}>\varphi(L_{t}^{0}(B))\right)=1-\exp\left(-\int_{0}^{+\infty}{\frac{dx}{\varphi(x)}}\right)
\end{equation}
and
\begin{equation}
	P\left(\exists t\leq\tau, \frac{|X_{t}|}{z_{\gamma_{t}}}>\varphi(L_{t}^{0}(B))\right)=1-\exp\left(-\int_{0}^{u}{\frac{dx}{\varphi(x)}}\right)
\end{equation}
That implies the following
$$P\left(\exists t\geq0, |X_{t}|>z_{\gamma_{t}}\varphi(L_{t}^{0}(B))\right)=1-\exp\left(-\int_{0}^{+\infty}{\frac{dx}{\varphi(x)}}\right)$$
and
$$P\left(\exists t\leq\tau, |X_{t}|>z_{\gamma_{t}}\varphi(L_{t}^{0}(B))\right)=1-\exp\left(-\int_{0}^{u}{\frac{dx}{\varphi(x)}}\right).$$
Consequently, the theorem is proved..
\end{proof}

\begin{coro}\label{c12}
Let $X$ be a process of the class $(\Sigma)$ which vanishes on $\{0\leq t\leq1: B_{t}=0\}$. Let $z$ be its associated predictable process. If for every $t\geq0$, $z_{\gamma_{t}}=1$, define $(\tau_{u})$ the right continuous inverse of $(L_{t}^{0}(B);t\geq0)$: 
$$ \tau_{u}=\inf{\{t\geq0:L_{t}^{0}(B)>u\}}.$$
Let $\varphi:\R_{+}\rightarrow\R_{+}$ be a Borel function. Then, we have the following estimates:
\begin{equation}
	P\left(\exists t\geq0, |X_{t}|>\varphi(L_{t}^{0}(B))\right)=1-\exp\left(-\int_{0}^{+\infty}{\frac{dx}{\varphi(x)}}\right)
\end{equation}
and
\begin{equation}
	P\left(\exists t\leq\tau_{u}, |X_{t}|>\varphi(L_{t}^{0}(B))\right)=1-\exp\left(-\int_{0}^{u}{\frac{dx}{\varphi(x)}}\right)
\end{equation}
\end{coro}

\begin{coro}\label{c13}
Let $X$ be a process of the class $(\Sigma)$ whose zero set coincides exactly with that of Brownian motion. If $\int_{0}^{+\infty}{\frac{dx}{\varphi}}=+\infty$, then the stopping time $T_{\varphi}=\inf\{t\geq0:\varphi(L_{t}^{0}(B))|X_{t}|>z_{\gamma_{t}}\}$ is finite almost surely. Furthermore, if $T_{\varphi}<\infty$ and $\varphi$ is locally bounded, then
$$|X_{T_{\varphi}}|=\frac{z_{\gamma_{T_{\varphi}}}}{\varphi(L_{T_{\varphi}}^{0}(B))}.$$
\end{coro}

\begin{coro}\label{c14}
Let $X$ be a process of the class $(\Sigma)$ which vanishes on $\{t\leq1: B_{t}=0\}$. If $z_{\gamma_{t}}\equiv1$ for all $t\geq0$ and $\int_{0}^{+\infty}{\frac{dx}{\varphi}}=+\infty$. Then the stopping time $T_{\varphi}=\inf\{t\geq0:\varphi(L_{t}^{0}(B))|X_{t}|>1\}$ is finite almost surely. Furthermore, if $T_{\varphi}<\infty$ and $\varphi$ is locally bounded, then
$$|X_{T_{\varphi}}|=\frac{1}{\varphi(L_{T_{\varphi}}^{0}(B))}.$$
\end{coro}

\begin{coro}\label{c15}
If $\int_{0}^{+\infty}{\frac{dx}{\varphi}}=+\infty$. Then the stopping time $T_{\varphi}=\inf\{t\geq0:\varphi(L_{t}^{0}(B))|B_{t}|>1\}$ is finite almost surely. Furthermore, if $T_{\varphi}<\infty$ and $\varphi$ is locally bounded, then
$$|B_{T_{\varphi}}|=\frac{1}{\varphi(L_{T_{\varphi}}^{0}(B))}.$$
\end{coro}

\subsection{Representation in terms of last passage time of Brownian motion}

In this subsection, we state results inspired by a representation formula obtained for stochastic processes of the class $(\Sigma)$ by Cheridito, Nikeghbali and Platen \cite{pat}. More precisely, they proved that under some assumptions, one has
$$X_{t}=E[X_{\infty}1_{\{g<t\}}|\mathcal{F}_{t}],$$
with, $g=\sup\{t\geq0: X_{t}=0\}$ and $X_{\infty}=\lim_{t\to+\infty}{X_{t}}$. Here, we shall show that some processes of the class $(\Sigma)$ vanishing on $\mathcal{Z}_{1}=\{t\leq1: B_{t}=0\}$ hold
 $$|X_{t}|=E[X_{\infty}1_{\{\gamma<t\}}|\mathcal{F}_{t}],$$
where, $\gamma=\sup\{t\geq0: B_{t}=0\}$ and $X_{\infty}=\lim_{t\to+\infty}{|X_{t}|}$.

Remark that $\gamma$ is not a stopping time with respect to the filtration $(\mathcal{F}_{t})_{t\geq0}$. Throughout of this subsection, we represent by $(\mathcal{G}_{t})_{t\geq0}$, the progressive enlargement of the filtration $(\mathcal{F}_{t})_{t\geq0}$ with respect to $\gamma$ and we consider that $R$ is the Azema sub-martingale associated with $\gamma$. More precisely, $R$ is defined by:
$$R_{t}=P\left(\gamma<t|\mathcal{F}_{t}\right).$$
\begin{prop}
If $X$ is a continuous stochastic process of the class $(\Sigma)$ which vanishes on $\{t\leq1; B_{t}=0\}$ such that $\langle X, R\rangle=0$. Then, $|X|R$ is a local martingale.
\end{prop}
\begin{proof}
We know that $R$ is a sub-martingale of the class $(\Sigma_{H}^{0})$. Hence, according to Proposition \ref{pm}, there exists a local martingale $m$ such that
$$R_{t}=m_{t}+\int_{0}^{t}{z_{s}dL_{s}^{0}(B)},$$
where $z$ is a predictable process with suitable integrability properties.

As far as, there exists a local martingale $M$ and a predictable process $z^{'}$ such that
$$|X_{t}|=M_{t}+\int_{0}^{t}{z^{'}_{s}dL_{s}^{0}(B)},$$
since, $|X|$ is a process of the class $(\Sigma)$ vanishing on the zero set of $B$. It follows from an integration by parts:
$$R_{t}|X_{t}|=\int_{0}^{t}{R_{s}d|X_{s}|}+\int_{0}^{t}{|X_{s}|dR_{s}}+\langle |X|, R\rangle_{t}.$$
Since $\langle |X|, R\rangle=0$, it entails that
$$R_{t}|X_{t}|=\int_{0}^{t}{R_{s}d|X_{s}|}+\int_{0}^{t}{|X_{s}|dR_{s}}$$
$$\hspace{5.5cm}=\int_{0}^{t}{R_{s}dM_{s}}+\int_{0}^{t}{R_{s}z^{'}_{s}dL_{s}^{0}(B)}+\int_{0}^{t}{|X_{s}|dm_{s}}+\int_{0}^{t}{X_{s}z_{s}dL_{s}^{0}(B)}.$$
But, the random measure $dL^{0}_{t}(B)$ is carried by the zero set of $B$. Then,
$$\int_{0}^{t}{R_{s}z^{'}_{s}dL_{s}^{0}(B)}=\int_{0}^{t}{|X_{s}|z_{s}dL_{s}^{0}(B)}=0,$$
since $R$ and $|X|$ vanish on the zero set of $B$.
Therefore,
$$R_{t}|X_{t}|=\int_{0}^{t}{R_{s}dM_{s}}+\int_{0}^{t}{|X_{s}|dm_{s}}.$$
Consequently, $(R_{t}|X_{t}|; t\geq0)$ is a local martingale.
\end{proof}

\begin{coro}\label{rep1}
Let $X$ be a process of the class $(\Sigma)$ vanishing on $\{t\leq1; B_{t}=0\}$ such that \\$\langle X, R\rangle=0$. If, $R|X|$ is a uniformly integrable martingale. Then, there exists a random variable $X_{\infty}$ such that $\lim_{t\to+\infty}{|X_{t}|}=X_{\infty}$ a.s and in $L^{1}$ and we have $\forall t\geq0$,
$$|X_{t}|=E\left[X_{\infty}1_{\{\gamma<t\}}|\mathcal{F}_{t}\right].$$
\end{coro}
\begin{proof}
We know by assumptions $R|X|$ is a uniformly integrable martingale. Thus, there exists a random variable $X_{\infty}$ such that 
$$X_{\infty}=\lim_{t\to+\infty}{R_{t}|X_{t}|}=\lim_{t\to+\infty}{|X_{t}|}.$$ 
But, we can see that $R|X|$ vanishes on $\{t\geq0; R_{t}=0\}$ and $\gamma=\sup\{t\geq0; R_{t}=0\}$. Hence, by an application of quotient theorem (Theorem 3.2 of \cite{1}. We recall it in Theorem \ref{qot} of Appendix), we obtain that $(|X_{t+\gamma}|;t>0)$ is a uniformly integrable martingale with respect to the filtration $(\mathcal{G}_{t+\gamma}; t>0)$. Thus, we have what follows
$$|X_{t+\gamma}|=E\left[X_{\infty}|\mathcal{G}_{t+\gamma}\right].$$
That implies the following equality
$$\rho(|X_{\cdot+\gamma}|)_{t}=\rho\left(E\left[X_{\infty}|\mathcal{G}_{\cdot+\gamma}\right]\right)_{t},$$
where $\rho$ is the function defined in Proposition 3.1 of \cite{1} (We recall it in Proposition \ref{rho} of Appendix).
But according to Proposition \ref{rho}, $|X_{t}|=\rho(|X_{\cdot+\gamma}|)_{t}$ since $|X|$ vanishes on $\{t\geq0; R_{t}=0\}$. Thus,
$$|X_{t}|=\rho\left(E\left[X_{\infty}|\mathcal{G}_{\cdot+\gamma}\right]\right)_{t}.$$
Now, consider the following optional process
$$Y_{t}=E\left[X_{\infty}1_{\{\gamma<t\}}|\mathcal{F}_{t}\right].$$
We can see that $Y$ vanishes on $\{t\geq0; R_{t}=0\}$ and $\forall t>0$,
$$Y_{t+\gamma}=E\left[X_{\infty}|\mathcal{F}_{t+\gamma}\right].$$
From Lemma (5,7) of \cite{jeulin80} (see Lemma \ref{j80} of appendix), the following equality holds:
$$\mathcal{G}_{t+\gamma}=\mathcal{F}_{t+\gamma}\text{, }\forall t>0.$$
Therefore,
$$Y_{t+\gamma}=E\left[X_{\infty}|\mathcal{G}_{t+\gamma}\right]=|X_{t+\gamma}|.$$
Consequently, it entails from uniqueness of Proposition \ref{rho} that
$$|X_{t}|=Y_{t}=E\left[X_{\infty}1_{\{\gamma<t\}}|\mathcal{F}_{t}\right].$$
\end{proof}

\begin{theorem}\label{rep2}
Let $X$ be an element of the class $(\Sigma D)$ which vanishes on $\{t\leq1; B_{t}=0\}$. Then, there exists a random variable such that $\lim_{t\to+\infty}{X_{t}}=X_{\infty}$ and $\forall t\geq0$,
$$X_{t}=E\left[X_{\infty}1_{\{\gamma<t\}}|\mathcal{F}_{t}\right].$$
\end{theorem}
\begin{proof}
From an application of quotient theorem, we can affirm that the process $\left(\frac{X_{t+\gamma}}{R_{t+\gamma}}; t>0\right)$ is a uniformly integrable martingale with respect to the filtration $(\mathcal{G}_{t+\gamma}; t>0)$. But, we remark that
$$R_{t+\gamma}=1\text{, }t>0.$$
Hence, it follows that $\left(X_{t+\gamma}; t>0\right)$ is a uniformly integrable martingale with respect to the filtration $(\mathcal{G}_{t+\gamma}; t>0)$. Then, we have what follows,
$$X_{t+\gamma}=E\left[X_{\infty}|\mathcal{G}_{t+\gamma}\right]\text{, }t>0.$$
Thus, it entails that
$$\rho(X_{\cdot+\gamma})_{t}=\rho\left(E\left[X_{\infty}|\mathcal{G}_{\cdot+\gamma}\right]\right)_{t}.$$
But according to Proposition \ref{rho}, $X_{t}=\rho(X_{\cdot+\gamma})_{t}$ since $X$ vanishes on $\{t\geq0; R_{t}=0\}$. Therefore,
$$X_{t}=\rho\left(E\left[X_{\infty}|\mathcal{G}_{\cdot+\gamma}\right]\right)_{t}.$$
Now, consider the following optional process
$$Y_{t}=E\left[X_{\infty}1_{\{\gamma<t\}}|\mathcal{F}_{t}\right].$$
We can see that $Y$ vanishes on $\{t\leq1; B_{t}=0\}=\{t\leq1; R_{t}=0\}$ and $\forall t>0$,
$$Y_{t+\gamma}=E\left[X_{\infty}|\mathcal{F}_{t+\gamma}\right].$$
From Lemma (5,7) of \cite{jeulin80}, the following equality holds:
$$\mathcal{G}_{t+\gamma}=\mathcal{F}_{t+\gamma}\text{, }\forall t>0.$$
Therefore,
$$Y_{t+\gamma}=E\left[X_{\infty}|\mathcal{G}_{t+\gamma}\right]=X_{t+\gamma}.$$
Consequently, we conclude from uniqueness of Proposition \ref{rho} that
$$X_{t}=Y_{t}=E\left[X_{\infty}1_{\{\gamma<t\}}|\mathcal{F}_{t}\right].$$
This completes the proof.
\end{proof}

\begin{coro}
Let $g<\infty$ be an honest time such that $\gamma\leq g$ almost surely. Then,
$$g\overset{law}{=}\gamma.$$
\end{coro}
\begin{proof}
We can see that $X_{t}=P(g<t|\mathcal{F}_{t})$ is a process of the class $(\Sigma)$ which vanishes on the zero set of $B$.
Hence, according to Theorem \ref{rep2}, one has
$$X_{t}=E\left[X_{\infty}1_{\{\gamma<t\}}|\mathcal{F}_{t}\right]=P(\gamma<t|\mathcal{F}_{t}).$$
Therefore,
$$P(g<t|\mathcal{F}_{t})=P(\gamma<t|\mathcal{F}_{t}).$$
That implies that,
$$E\left(P(g<t|\mathcal{F}_{t})\right)=E\left(P(\gamma<t|\mathcal{F}_{t})\right),$$
where, $E$ is the expectation.
Consequently, we have
$$P(g<t)=P(\gamma<t).$$
This completes proof.
\end{proof}

\section{Appendix}

\subsection{Enlargement of filtrations theory}
In this subsection, we recall results of theory of enlargement filtrations which were mainly  useful in the current work.
\begin{defn}[\textbf{Definition 2.1 of Azéma and Yor\cite{1}}]
Let $H$ be a random  optional closed set. We call $\mathcal{R}(H)$ the class of processes $(X_{t};t\geq0)$ vanishing on $H$ and admitting a decomposition of the form
$$X_{t}=M_{t}+V_{t},$$
where $(M_{t};t\geq0)$ is a right continuous uniformly integrable martingale, $(V_{t};t\geq0)$ is a continuous and adapted variation integrable process such that $dV_{t}$ is carried by $H$.
\end{defn}

\begin{prop}[\textbf{Proposition 3.1 of Azéma and Yor\cite{1}}]\label{rho}

Let $H$ be a random  optional closed set. Denote $g=\sup{H}$ and represent by $(\mathcal{G}_{t})_{t\geq0}$, the progressive enlargement of the filtration $(\mathcal{F}_{t})_{t\geq0}$ with respect to $g$. Let $(V_{t})_{t\geq0}$ be a $(\mathcal{G}_{g+t})_{t\geq0}-$   optional process. There exists a unique $(\mathcal{F}_{t})_{t\geq0}-$   optional process $(U_{t})_{t\geq0}$  which vanishes on $H$ such that $\forall t\geq0$, $U_{g+t}=V_{t}$ and $U_{0}=V_{0}$ on $\{g=0\}$. That defines a function
$\rho:V\longmapsto U$. $\rho$ is linear, non-negative and preserves products.
\end{prop}       

\begin{theorem}\label{qot}[\textbf{Quotient theorem: Theorem 3.2 of Azéma and Yor \cite{1}}]
\begin{enumerate}
	\item If $(X_{t};t\geq0)$ is a stochastic process of the class $\mathcal{R}(H)$, hence, the process $(\chi_{t};t>0)$ defined by 
$$\chi_{t}=\frac{X_{g+t}}{Y_{g+t}}$$
is a $\left(Q, (\mathcal{G}_{g+t}){t>0}\right)$ uniformly integrable martingale.
\item Reciprocally, let $(\chi_{t};t>0)$ a $\left(Q, (\mathcal{G}_{g+t}){t>0}\right)$ uniformly integrable martingale; the stochastic process $X=(Y_{t}\rho(\chi_{\cdot})_{t};t\geq0)$ is the unique process of $\mathcal{R}(H)$ such that
$$\chi_{t}=\frac{X_{g+t}}{Y_{g+t}}$$
for all $t>0$.
\end{enumerate}
\end{theorem}

\begin{lem}[\textbf{Lemma 5.7 of Jeulin \cite{jeulin80}}]\label{j80}
Let $g$ be an honest variable with respect to $(\mathcal{F}_{t})_{t\geq0}$. Let $(\mathcal{G}_{t})_{t\geq0}$ be the progressive enlargement of the filtration $(\mathcal{F}_{t})_{t\geq0}$ with respect to $g$. If $\tau$ is a stopping time with respect to $(\mathcal{G}_{t})_{t\geq0}$ such that $g<\tau$ on $\{g<\infty\}$, hence
$$\mathcal{G}_{\tau}=\mathcal{F}_{\tau}.$$
\end{lem}

\subsection{Balayage formula}
The balayage formulas in the progressive case were more used throughout this paper. In what follows, we recall results we used.

\begin{prop}[\textbf{Proposition 2.3 of Ouknine and Bouhadou \cite{siam}} ]
Let $Y$ be a continuous semimartingale and $\gamma_{t}=\sup\{s\leq t:Y_{s}=0\}$. If $k$ is a bounded progressive process, then
$$k_{\gamma_{t}}Y_{t}=k_{0}Y_{0}+{\int_{0}^{t}{^{p}k_{\gamma_{s}}dY_{s}}+R_{t}},$$
where $R$ is a process of bounded variation, adapted, continuous such that $dR_{t}$ is carried by the set $\{Y_{s}=0\}$.
\end{prop}

\begin{prop}[\textbf{Proposition 2.2 of Ouknine and Bouhadou \cite{siam}}]
Let $Y$ be a continuous semimartingale and $Z^{\alpha}$, the process defined in \eqref{zalpha}. Then,
$$Z^{\alpha}_{t}Y_{t}=\int_{0}^{t}{Z^{\alpha}_{s}dY_{s}}+(2\alpha-1)L_{t}^{0}(Z^{\alpha}Y),$$
where, $L_{t}^{0}(Z^{\alpha}Y)$ is the local time of the semimartingale $Z^{\alpha}Y$.
\end{prop}
{\color{myaqua}

}}


\begin{thebibliography}{10}
{\color{black}

\bibitem{1}
J.~Azéma and M.~Yor.
\newblock Sur les zéros des martingales continues. \newblock {\em Séminaire de probabilités (Strasbourg)}, 26: 248-306, 1992.
\bibitem{2}
J.~Azéma and M.~Yor.
\newblock Une solution simple au problème de Skorokhod. \newblock {\em in: Sém.proba. XIII, in: Lecture Notes in
Mathematics}, 721: 90-115,625-633, 1979.

\bibitem{siam}
S.~Bouhadou and Y.~Ouknine.
\newblock On the time inhomogeneous skew Brownian motion. \newblock {\em Bulletin des Sciences Mathématiques}, vol.137(7): 835-850, 2013.

\bibitem{pat}
P.~Cheridito, A.~Nikeghbali and E.~Platen.
\newblock Processes of class sigma, last passage times, and drawdowns. \newblock {\em SIAM Journal on Financial Mathematics}, 3(1): 280-303, 2012

\bibitem{eomt}
F.~Eyi-Obiang, Y.~Ouknine and O.~Moutsinga, G.~Trutnau.
\newblock Some contributions to the study of stochastic processes of the classes  \texorpdfstring{$\Sigma(H)$}{} and \texorpdfstring{$(\Sigma)$}{}. \newblock {\em Stochastics, 89, 8: 1253-1269, 2017.}

\bibitem{jeulin80}
T.~Jeulin.
\newblock Semi-martingales et grossissement d'une filtration. \newblock {\em Lecture Notes in Mathematics 1980, Springer}

\bibitem{man}
R.~Mansury, M.~Yor.
\newblock Random Times and Enlargements of Filtrations in a Brownian Setting. \newblock {\em Lecture Notes in Mathematics 1873, Springer, 2006, ISBN 3540294074, DOI 10.1007/11415558}
\bibitem{naj}
J.~Najnudel, A.~Nikeghbali.
\newblock A new construction of the $\Sigma$- finite measures associated with sub-martingales of class $(\Sigma)$. \newblock {\em C.R. Math. Acad. Sci. Paris}, 348: 311-316, 2010.
\bibitem{naj1}
J.~Najnudel, A.~Nikeghbali.
\newblock A remarkable sigma-finite measure associated with last passage times and penalisation results. \newblock {\em Contemporary Quantitative Finance, Essays in Honour of Eckhard Platen,Springer}, 77-98, 2010.


\bibitem{naj2}
J.~Najnudel, A.~Nikeghbali.
\newblock On some properties of a universal sigma finite measure associated with a remarkable class of sub-martingales. \newblock {\em Publ. of the Res. Instit. for Math. Sci. (Kyoto University)}, 47(4): 911-936, 2011.

\bibitem{naj3}
J.~Najnudel, A.~Nikeghbali.
\newblock On some universal sigma-finite measures and some extensions of Doob's optional stopping theorem. \newblock {\em Accepted in Stochastic processes and their applications}.

\bibitem{nik}
A.~Nikeghbali.
\newblock A class of remarkable sub-martingales. \newblock {\em Journal of Theoretical Probability}, 4(19): 931-949, 2006.
\bibitem{mult}
A.~Nikeghbali.
\newblock Multiplicative decompositions and frequency of vanishing of non negative sub-martingales. \newblock {\em Journal of Theoretical Probability}, 19(4): 931-949, 2006.

\bibitem{prok}
V.~Prokaj.
\newblock Unfolding the Skorokhod reflection of a semimartingale. \newblock {\em Statistics and probability Letters}, 79:534-536, 2009.

\bibitem{y}
M.~Yor.
\newblock Sur le balayage des semi-martingales continues. \newblock {\em Séminaire de probabilités (Strasbourg)}, 13: 453-471, 1979.
\bibitem{y1}
M.~Yor.
\newblock Les inégalités de sous-martingales, comme conséquences de la relation de domination. \newblock {\em Stochastics}, 3(1): 1-15, 1979.


}

\end{thebibliography}
\end{document}